\documentclass[reqno]{amsart}
\usepackage{fullpage}
\usepackage{stmaryrd}
\usepackage{graphicx}
\usepackage{hyperref}
\usepackage{enumitem}
\usepackage{amsmath,amssymb,amsfonts,mathrsfs,amscd}
\usepackage{newtxtext,newtxmath}
\usepackage[all,cmtip]{xy}

\newtheorem{thm}{Theorem}[section]
\newtheorem{lem}[thm]{Lemma}
\newtheorem{cor}[thm]{Corollary}

\newtheorem{thmA}{Theorem}

\newtheorem{corA}[thmA]{Corollary}

\theoremstyle{definition}

\theoremstyle{remark}

\normalsize

\numberwithin{equation}{section}

\newcommand{\NM}{\vartriangleleft}

\DeclareMathOperator{\I}{I_\pi}

\DeclareMathOperator{\Irr}{Irr}
\DeclareMathOperator{\IBr}{IBr}

\DeclareMathOperator{\CC}{\bf C}
\DeclareMathOperator{\NN}{\bf N}
\DeclareMathOperator{\OO}{\bf O}
\DeclareMathOperator{\dz}{dz}

\DeclareMathOperator{\XX}{X}

\begin{document}
\title{Weights for $\pi$-partial characters of $\pi$-separable groups}
\author{\sc Xuewu Chang
\and Ping Jin}
\address{School of Mathematical Sciences, Shanxi University, Taiyuan, 030006, China}
\email{changxuewu@sxu.edu.cn}
\email{jinping@sxu.edu.cn}
\date{}
%\thanks{Supported by NSF of China (12171289) and the NSF of Shanxi Province (20210302123429 and 20210302124077).}

\begin{abstract}
The aim of this paper is to confirm an inequality predicted by Isaacs and Navarro in 1995,
which asserts that for any $\pi'$-subgroup $Q$ of a $\pi$-separable group $G$, the number of $\pi'$-weights of $G$ with $Q$ as the first component always exceeds that of irreducible $\pi$-partial characters of $G$ with $Q$ as their vertex.
We also give some sufficient condition to guarantee that these two numbers are equal,
and thereby strengthen their main theorem on the $\pi$-version of the Alperin weight conjecture.
\end{abstract}

\maketitle

\Large

%%%----------------------------------------------------------
\section{Introduction}
Let $p$ be a prime, and let $G$ be a finite group.
Recall that a {\bf $p$-weight} of $G$ is a pair $(Q,\tau)$, where $Q$ is a $p$-subgroup of $G$ and
$\tau\in\Irr(\NN_G(Q)/Q)$ has $p$-defect zero, i.e., whose degree has the same $p$-part as $|\NN_G(Q):Q|$.
The Alperin weight conjecture (often abbreviated AWC) asserts that
the number of $G$-conjugacy classes of $p$-weights coincides with the number of conjugacy classes of $p'$-elements of $G$.

The first complete published proof of the AWC for $p$-solvable groups was given by
Isaacs and Navarro \cite{IN1995} in 1995 by introducing many new techniques to the $\pi$-theory of characters of $\pi$-separable groups for a set $\pi$ of primes, and in fact, they established a $\pi$-version of this conjecture.
Specifically, they defined a {\bf $\pi'$-weight} of a $\pi$-separable group $G$ to be an ordered pair $(Q,\tau)$,
where $Q$ is a $\pi'$-subgroup of $G$ and $\tau\in\Irr(\NN_G(Q)/Q)$ has $p$-defect zero for every prime $p\in\pi'$,
and then proved the following result, which appears in \cite{IN1995} as Theorem A
and reduces to the $p$-solvable case of AWC by taking $\pi=p'$, the complement of the prime $p$.

\begin{thm}[Isaacs-Navarro]\label{IN1}
Let $G$ be a $\pi$-separable group, and assume that the Hall $\pi$-complements of $G$ are nilpotent.
Then the number of $G$-conjugacy classes of $\pi'$-weights is equal to the number of conjugacy classes of $\pi$-elements of $G$.
\end{thm}

Furthermore, they defined the {\bf vertices} for every irreducible {\bf $\pi$-partial character} of a $\pi$-separable group $G$,
and showed that the set of vertices for an irreducible $\pi$-partial character of $G$ forms a single conjugacy class of $\pi'$-subgroups of $G$. We will review the relevant definitions and properties in the next section.
Now let $Q$ be a $\pi'$-subgroup of $G$, and write $\I(G|Q)$ for the set of irreducible $\pi$-partial characters of $G$
with vertex $Q$. It is easy to show that $|\I(\NN_G(Q)|Q)|$ is exactly the number of $\pi'$-weights of $G$ with first component $Q$; see Lemma 6.28 of \cite{I2018} or Lemma \ref{6.28} below.
Using the notion of vertices, Isaacs and Navarro obtained a strengthened version of Theorem \ref{IN1}, and they also weakened the nilpotency condition on $Q$; see Theorem C and Theorem 6.3 of \cite{IN1995}.

\begin{thm}[Isaacs-Navarro]\label{IN2}
Let $G$ be $\pi$-separable, and let $Q\le G$ be a solvable $\pi'$-subgroup with the property that
whenever $Q\le X<Y\le G$, where $Y$ is a $\pi'$-subgroup, we have $X<\NN_Y(X)$.
Then $|\I(G|Q)|=|\I(\NN_G(Q)|Q)|$.
\end{thm}

The main purpose of this paper is to confirm an inequality predicted by Isaacs and Navarro
at the end of \cite{IN1995}, which asserts that for any $\pi'$-subgroup $Q$ of a $\pi$-separable group $G$,
the number of $\pi'$-weights of $G$ with $Q$ as the first component always exceeds that of irreducible $\pi$-partial characters of $G$ with $Q$ as their vertex or, equivalently, that $|\I(\NN_G(Q)|Q)|\ge |\I(G|Q)|$.

%Now the main result in this paper can be stated as follows.

\begin{thmA}\label{thm:A}
Let $G$ be a $\pi$-separable group, and let $Q$ be a $\pi'$-subgroup of $G$. Then $$|\I(G|Q)|\le |\I(\NN_G(Q)|Q)|.$$
\end{thmA}

Our proof of Theorem \ref{thm:A} was inspired by the techniques developed by Navarro and Sambale in \cite{NS2021}.
Actually, they explored nilpotent $\pi'$-weights in $\pi$-separable groups
and gave a generalization of Theorem \ref{IN1}.

Using the same argument as in the proof of Theorem  \ref{thm:A},
we obtain a strong form of Theorem \ref{IN2}.

\begin{thmA}\label{thm:B}
Let $G$ be a $\pi$-separable group, and suppose that $Q\le G$ is a $\pi'$-subgroup
with the property that $Q<\NN_X(Q)$ for every $\pi'$-subgroup $X$ of $G$ that properly contains $Q$.
Then $$|\I(G|Q)|=|\I(\NN_G(Q)|Q)|.$$
\end{thmA}

The following is an immediate consequence of Theorem  \ref{thm:B}
in which $\XX_\pi(G)$ denotes the set of {\bf $\pi$-special} characters of $G$
(for the definition, see Section 2 below).

\begin{corA}\label{cor:C}
Let $G$ be a $\pi$-separable group and let $Q$ be a Hall $\pi$-complement of $G$. Then
\begin{enumerate}[label={\upshape(\alph*)}]
\item $|\I(G|Q)|=|\I(\NN_G(Q)|Q)|$, and
\item $|\XX_\pi(G)|=|\Irr(\NN_G(Q)/Q)|$.
\end{enumerate}
\end{corA}

We mention that part (b) in the above corollary is a deep result due to Wolf,
which appears in \cite{W1990} as Corollary 1.16.

All groups considered in this paper are finite, and the notation and terminology are mostly taken from Isaacs' books
\cite{I1976} and \cite{I2018}.

%%%---------------------------------------------
\section{Preliminaries}
In this section, we will review some necessary materials and most of them are taken from \cite{I2018}.

\subsection{$\pi$-Partial characters and vertices}
Let $G$ be a $\pi$-separable group for a set $\pi$ of primes, and let $G^0$ be the set of $\pi$-elements of $G$.
For any complex character $\chi$ of $G$, we say that the restriction $\chi^0$ of $\chi$ on $G^0$ is a {\bf $\pi$-partial character} of $G$. If a $\pi$-partial character $\varphi$ of $G$ cannot be written as a sum of two $\pi$-partial characters of $G$, then we say that $\varphi$ is an {\bf irreducible $\pi$-partial character} of $G$.
We use $\I(G)$ to denote the set of irreducible $\pi$-partial characters of $G$.
Note that if $\pi=p'$, the complement of a single prime $p$, then $\I(G)=\IBr_p(G)$ by the Fong-Swan theorem for $p$-solvable groups, and if $\pi$ is the set of all primes, then $\I(G)=\Irr(G)$.

Recall that an irreducible complex character $\chi\in\Irr(G)$ is \textbf{$\pi$-special} if $\chi(1)$ is a $\pi$-number and the determinantal order $o(\theta)$ is a $\pi$-number, where $\theta$ is any irreducible constituent of $\chi_S$ for every subnormal subgroup $S$ of $G$. Denote by $\XX_\pi(G)$ for the set of all $\pi$-special characters of $G$.

The following is essential in the study of $\pi$-partial characters.
\begin{lem}[{\cite[Theorem 3.14]{I2018}}]\label{3.14}
Let $G$ be $\pi$-separable. Then the map $\chi\mapsto\chi^0$
defines a bijection from $\XX_\pi(G)$ onto the set of those irreducible $\pi$-partial characters of $G$ whose degree is a $\pi$-number.
\end{lem}

Now let $Q$ be a $\pi'$-subgroup of $G$.
We say that $Q$ is a \textbf{vertex} for an irreducible $\pi$-partial character $\varphi\in\I(G)$ if there exists a subgroup $U$ of $G$ and $\theta\in\I(U)$, such that $\varphi=\theta^G$, $\theta(1)$ is a $\pi$-number and $Q$ is a Hall $\pi'$-subgroup of $U$.  As mentioned in the introduction, we write $\I(G|Q)$ for the set of irreducible $\pi$-partial characters of $G$ with vertex $Q$.

\begin{lem}[{\cite[Corollary 5.18]{I2018}}]\label{5.18}
Let $G$ be $\pi$-separable and let $\varphi\in\I(G)$.
If $\varphi\in\I(G|Q)$, then $\varphi(1)_{\pi'}|Q|=|G|_{\pi'}$.
In particular, if $Q$ is a Hall $\pi'$-subgroup of $G$, then
$\varphi\in\I(G|Q)$ if and only if $\varphi(1)$ is a $\pi$-number.
\end{lem}

Also, we say that $\chi\in\Irr(G)$ has {\bf $\pi'$-defect zero} if $\chi(1)_{\pi'}=|G|_{\pi'}$
or equivalently, if $\chi$ has $p$-defect zero for every $p\in\pi'$.
Write $\dz(G)$ for the set of irreducible complex characters of $G$ having $\pi'$-defect zero,
so a pair $(Q,\tau)$ is a $\pi'$-weight of $G$ if $Q$ is a $\pi'$-subgroup of $G$ and $\tau\in\dz(\NN_G(Q)/Q)$.

We need two basic facts on vertices and weights.

\begin{lem}[{\cite[Theorem 5.17]{I2018}}]
Let $G$ be a $\pi$-separable group and let $\varphi\in\I(G)$. Then all vertices for $\varphi$ is a single conjugacy class of $\pi'$-subgroups of $G$.
\end{lem}

\begin{lem}[{\cite[Lemma 6.28]{I2018}}]\label{6.28}
Let $G$ be a $\pi$-separable group and let $Q$ be a $\pi'$-subgroup of $G$. Then the number of $\pi'$-weights of $G$ with first component $Q$ is equal to $|\I(\NN_G(Q)|Q)|$.
Actually, $\tau\mapsto\tau^0$ defines a canonical bijection
$\dz(\NN_G(Q)/Q)\to\I(\NN_G(Q)|Q)$.
\end{lem}

%%%-------------------------------------------------------------
\subsection{The Clifford correspondence for $\pi$-partial characters}
Let $K\le G$ and $\theta\in\I(K)$. We say that $\varphi\in\I(G)$ \textbf{lies over} $\theta$
or equivalently, that $\theta$ \textbf{lies under} $\varphi$ if $\theta$ is an irreducible constituent of $\varphi_K$.
We use $\I(G|\theta)$ to denote the set of all irreducible $\pi$-partial characters of $G$ lying over $\theta$.
It is natural to define $\theta^G$ by the usual formula for induced characters but applied only to $\pi$-elements of $G$,
so that $\theta^G$ is also a $\pi$-partial character of $G$.

Furthermore, if $K\NM G$, then $\theta^g\in\I(K)$ for all $g\in G$,
where $\theta^g$ is defined by $\theta^g(x^g)=\theta(x)$ for $x\in K^0$.
Then $G$ acts on the set $\I(K)$ via conjugation, and we write $G_\theta$ for the stabilizer of $\theta$ in $G$.

Now we can state the Clifford correspondence for $\pi$-partial characters.

\begin{lem}[{\cite[Theorem 5.11]{I2018}}]\label{5.11}
Suppose that $\theta\in \I(N)$, where $N\NM G$ and $G$ is $\pi$-separable.
Then induction $\alpha\mapsto\alpha^G$ defines a bijection $\I(G_\theta|\theta)\to\I(G|\theta)$.
Also, if $\varphi=\alpha^G$, then $\alpha$ is the unique irreducible $\pi$-partial character of $G_\theta$
that lies under $\varphi$ and over $\theta$, and in this case, we often write $\alpha=\varphi_\theta$
and call it the Clifford correspondent of $\varphi$ over $\theta$.
\end{lem}

Following Navarro-Sambale \cite{NS2021}, we write $\I(G|Q,\tau)=\I(G|Q)\cap \I(G|\tau)$,
where $Q$ is a $\pi'$-subgroup of a $\pi$-separable group $G$ and $\tau\in\I(N)$ for some $N\NM G$.

\begin{lem}[{\cite[Lemma 6.33]{I2018}}]\label{6.33}
Let $G$ be a $\pi$-separable group and let $\varphi\in\I(G|Q)$, where $Q$ is a $\pi'$-subgroup of $G$,
and let $K\NM G$.
\begin{enumerate}[label={\upshape(\alph*)}]
\item There exists a unique $\mathbf{N}_G(Q)$-orbit of some $\theta\in\I(K)$ such that $\varphi_\theta\in\I(G_\theta|Q)$.
\item If $\tau\in\I(K)$ is $Q$-invariant and $G=G_\tau\NN_G(Q)$,
then induction of $\pi$-partial characters defines a bijection
$\I(G_\tau|Q,\tau)\to \I(G|Q,\tau)$.
\end{enumerate}
\end{lem}

The following corollary tells us how to count the quantity $|\I(G|Q)|$  ``over a normal subgroup".
\begin{cor}\label{cor}
Let $G$ be a $\pi$-separable group with $K\NM G$, and suppose that $Q$ is a $\pi'$-subgroup of $G$.
Let $\mathscr{A}$ be a set of representatives for the $\NN_G(Q)$-orbits
of $Q$-invariant members of $\I(K)$. Then
$$|\I(G|Q)|=\sum_{\tau\in\mathscr A}|\I(G_\tau|Q,\tau)|.$$
\end{cor}

Sometimes we need to work in quotient groups modulo normal $\pi'$-subgroups when studying $\pi$-partial characters and vertices.

\begin{lem}[{\cite[Theorem 3.10 and Lemma 5.31]{I2018}}]\label{3.10}
Let $G$ be a $\pi$-separable group and let $N \NM G$ be a $\pi'$-subgroup.
Write $\bar G=G/N$. Then
\begin{enumerate}[label={\upshape(\alph*)}]
\item There is a bijection $\varphi \mapsto \bar{\varphi}$ from $\I(G)$ onto $\I(\bar G)$, where $\bar{\varphi}(\bar x) = \varphi(x)$ for all $x \in G^0$.
\item If $\varphi \in \I(G|Q)$, where $Q$ is a $\pi'$-subgroup of $G$, then $N \le Q$ and $\bar\varphi\in\I(\bar G|\bar Q)$.
In particular, we have ${\bf O}_{\pi'}(G)\le Q$.
\end{enumerate}
\end{lem}

\begin{lem}[{\cite[Theorem 6.31]{I2018}}]\label{6.31}
Let $G$ be a $\pi$-separable group and let $K \NM G$ be a $\pi$-subgroup.
Also, let $Q\le G$ be a $\pi'$-subgroup, and let $L\NM G$, where $L\le Q$.
Let $\theta\in\Irr(K)$ be $G$-invariant, and write $\bar G=G/L$.
Then $|\I(G|Q,\theta)|=|\I(\bar G|\bar Q,\bar\theta)|$,
where $\bar\theta\in\Irr(\bar K)$ corresponds to $\theta$ via the natural isomorphism $K\to KL/L$.
\end{lem}

%%%---------------------------------------------
\subsection{Vertices and character-triple isomorphisms}
By a \textbf{character triple} $(G,N,\theta)$, we mean that $G$ is a group, $N\NM G$ and $\theta\in\Irr(N)$ is $G$-invariant. For the definition and properties of character-triple isomorphisms, see Chapter 1 of \cite{I2018} or Chapter 11 of \cite{I1976}.

\begin{lem}[{\cite[Lemma 3.11]{I2018}}]\label{3.11}
Let $(G, N, \theta)$ be a character triple, where $G$ is a group and $N$ is a $\pi$-subgroup.
Then there exists a character triple $(G^{*}, N^{*}, \theta^{*})$ such that $N^{*} \le \mathbf{Z}(G^{*})$ is also a $\pi$-subgroup and $(G^{*}, N^{*}, \theta^{*})$ is isomorphic to $(G, N, \theta)$.
\end{lem}

The irreducible $\pi$-partial characters behave very well under the isomorphism of character triples.
\begin{lem}[{\cite[Lemmas 6.21 and 6.32]{I2018}}]\label{6.21}
Let $(G, N, \theta)$ and $(G^{*}, N^{*}, \theta^{*})$ be isomorphic character triples, where $G$ and $G^{*}$ are $\pi$-separable and $N$ and $N^{*}$ are $\pi$-groups. Then
\begin{enumerate}[label*=\upshape(\alph*)]
\item There exists a unique bijection $\varphi \mapsto \varphi^*$ from $\I(G|\theta)$ onto $\I(G^*|\theta^*)$ such that if $\chi$ is a lift of $\varphi$ in $\Irr(G|\theta)$ and $\chi^* \in \Irr(G^*|\theta^*)$ corresponds to $\chi$ under the character-triple isomorphism, then $\chi^*$ is a lift of $\varphi^*$. In particular, we have $\varphi(1)_{\pi'}=\varphi^*(1)_{\pi'}$.
\item Let $Q \le G$ and $Q^* \le G^{*}$ be $\pi'$-groups, and assume that the intermediate subgroups $NQ$ and $N^{*}Q^{*}$ correspond via the character-triple isomorphism.
Then $\varphi\in\I(G|Q,\theta)$ if and only if  $\varphi^*\in\I(G^*|Q^*,\theta^*)$.
In particular, we have $|\I(G|Q,\theta)|=|\I(G^*|Q^*,\theta^*)|$.
\end{enumerate}
\end{lem}

%%%---------------------------------------------
\subsection{Above the Glauberman-Isaacs correspondence}
The following deep theorem was proved by Dade and Puig in the case where $Q$ is solvable,
which is thoroughly explained in \cite{T2008} as mentioned at the end of page 2532 of \cite{NS2021}.
If $Q$ is not solvable, then it must have even order by the Feit-Thompson theorem,
and thus $K$ has odd order. In this case, Lewis \cite{L1997} showed that the result also holds.

\begin{lem}[Dade-Puig-Lewis]\label{DP}
Let $G$ be an arbitrary group with $K\NM G$ and $Q\le G$, and assume that $KQ\NM G$ and that $(|K|,|Q|)=1$.
Let $\theta\in\Irr(K)$ be $Q$-invariant, and write $\theta^*\in\Irr(\CC_K(Q))$
for its Glauberman-Isaacs correspondent.
Then $(G_\theta,K,\theta)$ and $(\NN_G(Q)_{\theta^*},\CC_K(Q),\theta^*)$ are isomorphic character triples.
\end{lem}

%%%---------------------------------------------
\subsection{Weights and vertices}
One of essential ingredients in our proofs of Theorems \ref{thm:A} and \ref{thm:B} is Theorem \ref{basic} below,
whose proof relies on the following lemma.

\begin{lem}[{\cite[Lemma 6.30]{I2018}}]\label{6.30}
Let $K \NM G$, where $G$ is $\pi$-separable and $K$ is a $\pi$-subgroup,
and suppose $Q \le G$ is a $\pi'$-subgroup such that $KQ \NM G$.
Let $\theta \in \Irr(K)$ be $Q$-invariant and $\varphi \in \I(G|\theta)$.
Then the following hold.
\begin{enumerate}[label={\upshape(\alph*)}]
\item There exists a vertex $P$ of $\varphi$ such that $Q\le P\le \NN_G(Q)$.
In particular, $Q$ is a vertex for $\varphi$ if and only if  $\varphi(1)_{\pi'}=|G:Q|_{\pi'}$.

\item Furthermore, any two of such vertices $P$ for $\varphi$ are $\NN_G(Q)$-conjugate.
\end{enumerate}
\end{lem}

\begin{thm}\label{basic}
Let $K \NM G$, where $G$ is $\pi$-separable and $K$ is a $\pi$-subgroup,
and suppose $Q \le G$ is a $\pi'$-subgroup such that $KQ \NM G$.
For any $Q$-invariant character $\tau\in\Irr(K)$, we have
$$|\I(G|Q,\tau)|=|\I({\bf N}_G(Q)|Q,\tau^*)|$$
where $\tau^*$ is the image of $\tau$ under the Glauberman-Isaacs correspondence
$\Irr_Q(K)\to \Irr({\bf C}_K(Q))$.
\end{thm}

\begin{proof}
By the Frattini argument, we have $G=K{\bf N}_G(Q)=G_\tau{\bf N}_G(Q)$,
and by Lemma \ref{6.33}(b), we see that $|\I(G|Q,\tau)|=|\I(G_\tau|Q,\tau)|$.
Also, since ${\bf N}_G(Q)_{\tau^*}={\bf N}_G(Q)_\tau$, the same reason yields
$$|\I({\bf N}_G(Q)|Q,\tau^*)|=|\I({\bf N}_G(Q)_\tau|Q,\tau^*)|
=|\I({\bf N}_{G_\tau}(Q)|Q,\tau^*)|.$$
To complete the proof, therefore, we can assume without loss that $\tau$ is $G$-invariant.
Then, by Lemma \ref{DP}, the character triples $(G,K,\tau)$ and $({\bf N}_{G}(Q),{\bf C}_K(Q),\tau^*)$ are isomorphic,
and by Lemma \ref{6.21}(a), the map $\varphi\mapsto\varphi^*$ defines a bijection $\I(G|\tau)\to\I({\bf N}_{G}(Q)|\tau^*)$ such that $\varphi(1)_{\pi'}=\varphi^*(1)_{\pi'}$.
In this case, Lemma \ref{6.30}(a) tells us that $Q$ is a vertex for $\varphi$ if and only if
$$\varphi(1)_{\pi'}=|G:Q|_{\pi'}=|{\bf N}_G(Q):Q|_{\pi'},$$
and this is equivalent to saying that $\varphi^*(1)_{\pi'}=|{\bf N}_G(Q):Q|_{\pi'}$.
Thus $\varphi\in\I(G|Q,\tau)$ if and only if $\varphi^*\in\I({\bf N}_{G}(Q)|Q,\tau^*)$,
which implies that the above map $\varphi\mapsto\varphi^*$ is also a bijection
from $\I(G|Q,\tau)$ onto $\I({\bf N}_{G}(Q)|Q,\tau^*)$, and the result follows.
\end{proof}

%%%-----------------------------------------------------------

\section{Proof of Theorems A and B}
In this section, we will establish Theorems \ref{thm:A} and \ref{thm:B} simultaneously from the introduction.
As usual, we need to prove the following ``relative version",
which clearly includes Theorems \ref{thm:A} and \ref{thm:B} (by taking $Z=1$).

\begin{thm}\label{thm}
Let $G$ be a $\pi$-separable group, and let $Q$ be a $\pi'$-subgroup of $G$.
Suppose that $Z$ is a normal $\pi$-subgroup of $G$ and that $\lambda\in\Irr(Z)$ is invariant in $G$.
Then the following hold.
\begin{enumerate}[label={\upshape(\alph*)}]
\item $|\I(G|Q,\lambda)|\le |\I(\NN_G(QZ)|Q,\lambda)|$.
\item Equality holds in (a) if $Q$ has the property that
whenever $Q<X$, where $X$ is a $\pi'$-subgroup of $G$,
then $Q<\NN_X(Q)$.
\end{enumerate}
\end{thm}

\begin{proof}
We proceed by induction on $|G:Z|$.

\smallskip
\emph{Step 1. We may assume that $Z\le {\bf Z}(G)$, so that $\NN_G(QZ)=\NN_G(Q)$.}

By Lemma \ref{3.11}, we can choose a character triple $(G^*,Z^*,\lambda^*)$
isomorphic to $(G,Z,\lambda)$, where $Z^*\le {\bf Z}(G^*)$ is a $\pi$-subgroup. Since $G^*/Z^*\cong G/Z$ and $G$ is $\pi$-separable, we see that $G^*$ is also $\pi$-separable.

For each $\pi'$-subgroup $X$ of $G$,
we know that the intermediate subgroup $XZ$ corresponds to $(XZ)^*$
via the isomorphism $XZ/Z\cong (XZ)^*/Z^*$,
and since $XZ/Z\cong X$ is a $\pi'$-group and $Z^*$ is central in $(XZ)^*$,
there exists a unique $\pi'$-subgroup $X^*$ of $G^*$ such that $(XZ)^*=X^*\times Z^*$.
It follows that $X\mapsto X^*$ defines a surjective map from the set of $\pi'$-subgroups of $G$
onto that of $G^*$.

For (a), we conclude by Lemma \ref{6.21}(b) that $|\I(G|Q,\lambda)|=|\I(G^*|Q^*,\lambda^*)|$ and $$|\I(\NN_G(QZ)|Q,\lambda)|=|\I(\NN_{G^*}(Q^*Z^*)|Q^*,\lambda^*)|.$$
So, it suffices to prove (a) in $G^*$.

For (b), if $Q<X$, then $QZ<XZ$, which implies that $Q^*Z^*=(QZ)^*<(XZ)^*=X^*Z^*$ and hence $Q^*<X^*$.
Using the natural isomorphism $X\cong XZ/Z$, it is easy to see that  $Q<\NN_X(Q)$ if and only if $QZ/Z<\NN_{XZ/Z}(QZ/Z)$,
and this is equivalent to saying that
$$(QZ)^*/Z^*<\NN_{(XZ)^*/Z^*}((QZ)^*/Z^*)=\NN_{X^*Z^*/Z^*}(Q^*Z^*/Z^*),$$
which again happens if and only if $Q^*<\NN_{X^*}(Q^*)$.
So, we may also work in $G^*$ to establish (b).
To complete the proof, therefore, we may assume without loss that $Z\le {\bf Z}(G)$.

\smallskip
\emph{Step 2. We may assume that ${\bf O}_{\pi'}(G)=1$.}

Let $L={\bf O}_{\pi'}(G)$ and write $\bar G=G/L$. Suppose first that $L\not\le Q$.
Then Lemma \ref{3.10}(b) tells us that $\I(G|Q,\lambda)$ is empty, and (a) follows.
For (b), we need to prove that $\I(N|Q,\lambda)$ is also empty.
To do this, write $N=\NN_G(Q)$, and since $L\cap N \le {\bf O}_{\pi'}(N)$, it suffices to show that $L\cap N\not\le Q$.
Note that $Q<LQ$, and by the assumption on $Q$, we have
$$Q<\NN_{LQ}(Q)=\NN_{L}(Q)Q=(L\cap N)Q$$
and thus $L\cap N\not\le Q$.
This proves the theorem in the case where $L\not\le Q$.

Now suppose that $L\le Q$. By Lemma \ref{6.31}, we can deduce that
       $$|\I(G|Q,\lambda)|=|\I(\bar G|\bar Q,\bar\lambda)|
       \quad\text{and}\quad |\I(\NN_G(Q)|Q,\lambda)|=|\I(\NN_{\bar G}(\bar Q)|\bar Q,\bar\lambda)|.$$
So, it is no loss to work in $\bar G=G/L$, and thus we may assume that $L=1$
(note that $\OO_{\pi'}(G/L)=1$).

\smallskip
\emph{Step 3. We may assume that $|G:K|<|G:Z|$, where $K=\OO_\pi(G)$.}

Note that $Z\le K$, so we need to show that $Z<K$.
Suppose that $Z=K$, so that $K$ is central in $G$ by Step 1.
Since $\OO_{\pi'}(G)=1$ by Step 2, the Hall-Higman Lemma 1.2.3 (see Theorem 3.21 of \cite{I2008} for example)
implies that $G=\CC_G(K)\le K$. It follows that $G$ is a $\pi$-group, and
in this case, the result is trivial. Thus we may assume that $Z<K$.

\smallskip
\emph{Step 4. Complete the proof.}

Let $\mathscr{A}$ be a set of representatives of $\NN_G(Q)$-orbits in the set of $Q$-invariant characters in $\Irr(K|\lambda)$.
Since $\lambda$ is $G$-invariant, it follows by Corollary \ref{cor} that
   $$|\I(G|Q,\lambda)|=\sum_{\tau\in\mathscr{A}}
|\I(G_\tau|Q,\tau)|.$$
Furthermore, let $\mathscr{A}^*$ be the image of $\mathscr{A}$
under the Glauberman-Isaacs correspondence from $\Irr_Q(K)$ to $\Irr({\bf C}_K(Q))$.
Then $\mathscr{A}^*$ is clearly a set of representatives of $\NN_G(Q)$-orbits in
$\Irr({\bf C}_K(Q)|\lambda)$.
As before, we have
$$|\I(\NN_G(Q)|Q,\lambda)|=\sum_{\tau\in\mathscr{A}}|\I(\NN_G(Q)_\tau|Q,\tau^*)|.$$
Note that $K\NN_G(Q)_\tau=\NN_{G_\tau}(KQ)$,
and by Theorem \ref{basic}, we have
$$|\I(\NN_{G_\tau}(KQ)|Q,\tau)|=|\I(\NN_G(Q)_\tau|Q,\tau^*)|,
$$
for all $\tau\in\mathscr A$. To complete the proof, therefore, it suffices to show that
\[ |\I(G_\tau|Q,\tau)|\le |\I(\NN_{G_\tau}(KQ)|Q,\tau)| \]
with equality if $Q$ has the property described in (b).
This follows by the inductive hypothesis applied in $G_\tau$ since then $|G_\tau:K|<|G:Z|$ by Step 4,
and the proof is complete.
\end{proof}

The following is Corollary \ref{cor:C} in the introduction, which we restate here for convenience.
\begin{cor}
Let $G$ be a $\pi$-separable group, and let $Q$ be a Hall $\pi$-complement of $G$. Then
\begin{enumerate}[label={\upshape(\alph*)}]
\item $|\I(G|Q)|=|\I(\NN_G(Q)|Q)|$ and
\item $|\XX_\pi(G)|=|\Irr(\NN_G(Q)/Q)|$.
\end{enumerate}
\end{cor}
\begin{proof}
 (a) is clear by Theorem \ref{thm:B}.
To prove (b), we use Lemmas \ref{3.14} and \ref{5.18} to deduce that
$$|\XX_\pi(G)|=|\I(G|Q)|\;\text{and}\;|\XX_\pi(\NN_G(Q))|=|\I(\NN_G(Q)|Q)|,$$
and by (a), we obtain $|\XX_\pi(G)|=|\XX_\pi(\NN_G(Q))|$.
By definition, it is easy to see that all normal $\pi'$-subgroups of a $\pi$-separable group are contained
in the kernel of each $\pi$-special character,
and thus we may identify $\XX_\pi(\NN_G(Q))$ with $\XX_\pi(\NN_G(Q)/Q)$.
But $\NN_G(Q)/Q$ is a $\pi$-group, so all irreducible characters of $\NN_G(Q)/Q$ are $\pi$-special,
that is, $\Irr(\NN_G(Q)/Q)=\XX_\pi(\NN_G(Q)/Q)$.
It follows that $|\XX_\pi(G)|=|\Irr(\NN_G(Q)/Q)|$, as desired.
\end{proof}

\section*{Acknowledgements}
This work was supported by the NSF of China (12171289) and
by Fundamental Research Programs of Shanxi Province (20210302123429 and 20210302124077).
%The authors would like to thank the referee for helpful comments and suggestions.

%%%%%%%%%%%%%%%%%%%%%%%%%%%%%%%%%%%%
\bibliographystyle{amsplain}

%%%%%%%%%%%%%%%%%%%%%%%%%%%%%%%%%%%%
\end{document}